\documentclass[a4paper,11pt]{article}
\usepackage{amsfonts,amssymb,amsmath,amsthm}

\usepackage[english]{babel}
\usepackage[T1]{fontenc} 
\usepackage[latin1]{inputenc}

\usepackage[]{geometry}
\usepackage{bbm}	
\usepackage{mathrsfs}	
\usepackage{stmaryrd}	
\usepackage{epic,srcltx}

\usepackage{mparhack}
\usepackage[english]{babel}
\newtheoremstyle{theoremstyle}
{\topsep}
{\topsep}
{\itshape}
{-0mm}
{\bfseries}
{.}
{0.5em}
{}

\theoremstyle{theoremstyle}
\newtheorem{theorem}{Theorem}
\newtheorem{lemma}[theorem]{Lemma}

\newtheorem{proposition}[theorem]{Proposition}
\newtheorem*{theorem*}{Theorem}
\newtheorem*{lemma*}{Lemma}
\newtheorem*{corollary*}{Corollary}
\newtheorem*{proposition*}{Proposition}
\newtheorem*{defandprop*}{Definition and Proposition}

\theoremstyle{theoremstyle}

\newtheorem*{definition*}{Definition}
\newtheorem*{notation*}{Notation}

\newtheoremstyle{proofstyle}
{0pt}
{5pt}
{}
{-0mm}
{\itshape}
{}
{0.5em}
{}

\theoremstyle{proofstyle}

\newtheoremstyle{remarkstyle}
{0pt}
{5pt}
{}
{}
{\sffamily}
{.}
{0.5em}
{}

\theoremstyle{remarkstyle}
\newtheorem*{remark}{Remark}

\newcommand{\N}{{\mathbb N}}
\newcommand{\Z}{{\mathbb Z}}

\renewcommand{\P}{{\mathbb P}}
\newcommand{\E}{{\mathbb E}}
\allowdisplaybreaks[1]

\newcommand{\ssup}[1] {{{\scriptscriptstyle{({#1}})}}}

\begin{document}
\title{Stable limit laws for the parabolic Anderson model between quenched and annealed behaviour\renewcommand{\thefootnote}{\arabic{footnote}}} 

\author{\renewcommand{\thefootnote}{\arabic{footnote}}
{\sc J\"urgen G\"artner}
\footnotemark[1]
\\
\renewcommand{\thefootnote}{\arabic{footnote}}
{\sc Adrian Schnitzler}
\footnotemark[1]
}
\footnotetext[1]{
Institut f\"ur Mathematik, Technische Universit\"at Berlin,
Stra{\ss}e des 17.\ Juni 136, 10623 Berlin, Germany,
{\sl jg@math.tu-berlin.de}, 
{\sl schnitzler@math.tu-berlin.de}
}
\footnotetext[2]{
The work was supported by the DFG International Research Training Group {\it Stochastic Models of Complex Processes
}}

\date{\today}

\maketitle

\begin{abstract}
We consider the solution to the parabolic Anderson model with homogeneous initial condition in large time-dependent boxes. We derive stable limit theorems, ranging over all possible scaling parameters, for the rescaled sum over the solution depending on the growth rate of the boxes. Furthermore, we give sufficient conditions for a strong law of large numbers. 


\end{abstract}

\section{Introduction}
\subsection{The problem}
The parabolic Anderson model (PAM) is the heat equation on the lattice with a random potential, given by
\begin{eqnarray}\label{AWPe}
\begin{cases}
\frac{\partial}{\partial t}u(t,x) = \kappa \Delta u(t,x)+ \xi(x) u(t,x),
 \qquad &(t,x)\in (0,\infty)\times\mathbb{Z}^{d},\\ 
 u(0,x)=u_{0}(x), \qquad &x\in \mathbb{Z}^{d},
\end{cases}
\end{eqnarray}
where $\kappa>0$\index{$\kappa$} denotes a diffusion constant, $u_{0}$ a nonnegative function, and $\Delta$ the discrete Laplacian, defined by	
	\[
\Delta f(x) := \sum\limits_{\substack{y\in\mathbb{Z}^{d}:\\|x-y|_1=1}}\left[f(y)-f(x)\right],\qquad x\in \mathbb{Z}^{d},\, f\colon\mathbb{Z}^{d}\rightarrow \mathbb{R}.
\]
Furthermore, $\xi:=\left\{\xi(x),x\in\Z^d\right\}$ is an i.i.d.\ random potential. We will stick in this paper to the homogeneous initial condition $u_0\equiv1$.\\
The solution $u$ depends on two effects. The Laplacian tends to make it flat whereas the potential causes the occurrence of small regions where almost all mass of the system is located. The latter effect is called {\em intermittency}. It turns out that it becomes the more dominant the more heavy tailed the potential tails are.\\
Basically, there are two ways of looking at the solution. On the one hand one can pick one realisation of the potential field and consider the almost sure behaviour of $u$. This is the so called {\em quenched} setting. On the other hand one can take expectation with respect to the potential and consider the averaged behaviour of $u$. This is the so-called {\em annealed} setting. Expectation with respect to $\xi$ will be denoted by $\left<\cdot\right>$, and the corresponding probability measure will be denoted by $\mathbf{P}$. Those realisations of $\xi$ that govern the quenched behaviour of $u$ differ heavily from those that govern the annealed behaviour, see \cite{GM90}. Therefore, it is interesting to understand the transition mechanism from quenched to annealed behaviour.  \\  
To this end we are interested in expressions such as $\frac{1}{|Q|}\sum_{x\in Q}u(t,x)$ where $Q$ is a large centred  box. If $Q$ has a fixed size then $\frac{1}{|Q|}\sum_{x\in Q}u(t,x)$ follows quenched behaviour as $t$ tends to infinity. This can be deduced from the Feynman-Kac representation of $u$ given by
 \[u(t,x)=\mathbb{E}_x\exp\Bigg\{\int\limits_0^t \xi\left(X_s\right)\,{\rm d}s\Bigg\}u_0\left(X_t\right),\qquad(t,x)\in [0,\infty)\times\mathbb{Z}^{d},\]
where $X$ is a simple, symmetric, continuous time random walk with generator $\kappa\Delta$ and $\P_x$ ($\mathbb{E}_x$) denotes the corresponding probability measure (expectation) if $X_0=x$ a.s. \\
On the other hand, if we fix $t$ and let the size of $Q$ tend to infinity then (due to the homogeneous initial condition) by Birkhoff's ergodic theorem $\frac{1}{|Q|}\sum_{x\in Q}u(t,x)$ displays annealed behaviour almost surely. Therefore, a natural question is what happens if the box $Q$ is time dependent. \\
More precisely, we want to find for all $\alpha\in(0,2)$ a large box $Q_{L_{\alpha}(t)}$, with $Q_{r(t)}=[-r(t),r(t)]^d\cap\Z^d$, for any $r(t)>0$, and numbers $A(t),B_{\alpha}(t)$ such that
\[\sum\limits_{x\in Q_{L_\alpha(t)}}\frac{u(t,x)-A(t)}{B_{\alpha}(t)}\stackrel{t\to\infty}{\Longrightarrow}\mathcal{F_\alpha},\]
with $\mathcal{F}_\alpha$ a suitable stable distribution.\\
In the case $\kappa=0$, i.e.\ if the solutions at different sites are independent, the problem has been addressed in \cite{BABM05} under the assumption that the logarithmic tail of the distribution is normalized regularly varying. A wider class of disributions was considered in \cite{B06}. In \cite{J10} a conceptual treatment of several classes of timedependent sums is offered, in particular explaining the universality of the limit laws in different cases. In \cite{BABM05} the authors also give sufficient and necessary conditions on the growth rate of $Q$ for a weak law of large numbers (WLLN) and for a central limit theorem (CLT) to hold. Corresponding results for a WLLN and a CLT for the PAM were derived in \cite{BAMR05} and in \cite{BAMR07}. They state that, under appropriate regularity assumptions, there exist $J(t)$ and $\gamma_1<\gamma_2$, all depending on the tails of $\xi$, such that:
\begin{enumerate}
 \item $\frac{1}{|Q_{\gamma J(t)}|}\sum\limits_{x\in Q_{\gamma J(t)}}u(t,x)\sim\left<u(t,0)\right>$, as $t\to\infty$ if $\gamma>\gamma_1$, in probability,\\
 $\frac{1}{|Q_{\gamma J(t)}|}\sum\limits_{x\in Q_{\gamma J(t)}}u(t,x)=o\left(\left<u(t,0)\right>\right)$, as $t\to\infty$ if $\gamma<\gamma_1$, in probability.
\item $\frac{1}{|Q_{\gamma J(t)}|}\sum\limits_{x\in Q_{\gamma J(t)}}\frac{u(t,x)-\left<u(t,0)\right>}{\sqrt{\left<u(t,0)^2\right>}}\Longrightarrow\mathcal{N}(0,1)$, as $t\to\infty$ if $\gamma>\gamma_2$,\\
$\frac{1}{|Q_{\gamma J(t)}|}\sum\limits_{x\in Q_{\gamma J(t)}}\frac{u(t,x)-\left<u(t,0)\right>}{\sqrt{\left<u(t,0)^2\right>}}=o(1)$, as $t\to\infty$ if $\gamma<\gamma_2$, in probability.
\end{enumerate}
\noindent
Here, $\mathcal{N}(0,1)$ denotes the law of the standard normal distribution with variance 1.\\
 However, $\alpha$-stable limits for the PAM have not been investigated so far. Furthermore, we give sufficient conditions on the growth rate of $Q$ for a strong law of large numbers to hold. So far this has been done neither for the PAM nor for the $\kappa=0$ case.\\
For a general overview of the parabolic Anderson model see for instance \cite{M94} and \cite{GK05}. A WLLN and a CLT for the PAM with time-dependent white noise potential using rather different techniques can be found in \cite{CM07}. Similar questions concerning a version of the random energy model were investigated in \cite{BKL02}.  

\subsection{Main results}
To state the main results we need to introduce some notation. Let
\[\varphi(h):=-\log\mathbf{P}(\xi(0)>h)\]
 and $h_t$ being a solution to
\[\sup\limits_{h\in(0,\infty)}\left(th-\varphi(h)\right)=th_t-\varphi\left(h_t\right).\]  
If $\varphi$ is ultimately convex then $h_t$ is unique for any large $t$.
Throughout this paper we will assume that 
$\xi(0)$ is unbounded from above and has finite exponential moments of all orders.
 Under these circumstances the left-continuous inverse of $\varphi$,
\[\psi(s):=\min\left\{r\colon \varphi(r)\geq s\right\},\qquad s>0, 
\]
is well defined. Furthermore, this implies that the cumulant generating function 
\[H(t):=\log\left<\exp\{t\xi(0)\}\right>,\qquad t\geq0,\]
 is well-defined and that $H(t)<\infty$ for all $t$ with $\lim_{t\to\infty}H(t)/t=\infty$. If $\varphi\in C^2$ is ultimately convex and satisfies some mild regularity assumptions then the Laplace method yields that $H(t)=th_t-\varphi(h_t)+o(t)$.
In the sequel we will frequently need the following regularity assumptions.\\
\noindent
\underline{Assumption F}:\\
There exists $\rho\in[0,\infty]$ such that for all $c\in(0,1)$, 
\[\lim_{t\to\infty}\left[\psi(ct)-\psi(t)\right]=\rho c\log c.\]
\noindent
\underline{Assumption H}:\\
There exists $\rho\in[0,\infty]$ such that for all $c\in(0,1)$,
\[\lim_{t\to\infty}\frac{1}{t}\left[H(ct)-cH(t)\right]=\rho c\log c.\]
In \cite[Theorems 1.2 and 2.2]{GM98} the authors prove that there exists $\chi=\chi(\rho)\in[0,2d\kappa]$ such that 
\begin{equation}\label{gm98}
\frac{\log u(t,0)}{t}=\xi^{\ssup 1}_{Q_{t}}-\chi+o(1),\qquad \text{a.s.},
\end{equation}
with $\xi^{\ssup 1}_{A}=\sup\{\xi(x)\colon x\in A\}$, if Assumption F is satisfied, and
\begin{equation}\label{GM98}
 \frac{\log\left<u(t,0)^p\right>}{t}=\frac{H(pt)}{t}-p\chi+o(1),\qquad p\in\N,
\end{equation}
if Assumption H is satisfied. Notice that Assumption F implies Assumption H. Furthermore, it turns out that $\chi=\chi(\rho)$ is strictly increasing in $\rho$ with $\chi(0)=0$ and $\chi(\infty)=2d\kappa$.  For details see \cite{GM98}.\\
Prominent examples satisfying Assumption F are the double exponential distribution, i.e. $\mathbf{P}(X>x)=\exp\{-\exp\{x/\rho\}\},\,x>0$, for $\rho\in(0,\infty)$ and the Weibull distribution, i.e. $\mathbf{P}(X>x)=\exp\{-x^\gamma\},\ x>0$ with $\gamma>1$ for $\rho=\infty$. \\
For $\alpha\in (0,2)$ let $\mathcal{F}_\alpha$ be the $\alpha$-stable distribution with characteristic function 
\[\phi_\alpha(u)=\begin{cases}
  \exp\left\{-\Gamma(1-\alpha)|u|^\alpha\exp\left\{\frac{-i\pi\alpha}{2}\text{sign} u\right\}\right\},&\alpha\neq1,\\
   \exp\left\{iu(1-\gamma)-\frac{\pi}{2}|u|\left(1+2\pi i\log |u|\text{ sign}u\cdot \right)\right\},\qquad &\alpha=1.            
              \end{cases}
\]
Moreover, let 
\[L_{\alpha}(t):=\exp\{\varphi\big(h_{\alpha t}\big)\} \text{   and   } B_{\alpha}(t):=\exp\{t\cdot\big(h_{\alpha t}-\chi+o(1)\big)\},\]
 where the error term of $B_{\alpha}(t)$ is chosen in suitable way. Then we find our main result: 

\begin{theorem}[Stable limit laws]\label{stablemain}
 Let $\varphi\in C^2$ be ultimately convex and Assumption F be satisfied. Then for $\alpha\in(0,2)$,
\[\sum\limits_{x\in Q_{L_{\alpha}(t)}}\frac{u(t,x)-A(t)}{B_{\alpha}(t)}\stackrel{t\to\infty}{\Longrightarrow}\mathcal{F_\alpha},\]
with
\[A(t)=\begin{cases}
        0, &\text{if } \alpha\in(0,1),\\
\left<u(t,0)\right>,  &\text{if } \alpha\in(1,2),\\
\left<u(t,0)\mathbbm{1}_{u(t,0)\leq B_{\alpha}(t)}\right>, \quad &\text{if } \alpha=1.     \end{cases}
\]
\end{theorem}
\noindent
Furthermore, we find: 

\begin{theorem}[Strong law of large numbers]\label{slln}
 Let Assumption H be satisfied, and $r(t)$ be so large that $\lim_{t\to\infty}\frac 1t\big(\log|Q_{r(t)}|-H(2t)+2H(t)\big)>0$ then for every sequence $(t_n)_{n\in\N}$ satisfying $\sum_{t_n}\exp\{-t_n\}<\infty$,
\[\frac{1}{|Q_{r(t_n)}|}\sum_{x\in Q_{r(t_n)}}\left(\frac{u(t_n,x)}{\left<u(t_n,0)\right>}-1\right)\stackrel{t_n\to\infty}{\longrightarrow} 0 \quad\text{ a.s.}\]
\end{theorem}

Notice that the necessary growth rate of $Q$ for a WLLN to hold is the same as in Theorem~\ref{stablemain} for $\alpha=1$ and that the necessary growth rate of $Q$ for a CLT to hold corresponds to $\alpha=2$, see \cite{BAMR07}. The growth rate in Theorem~\ref{slln} is of the same order as in the CLT case. Notice, that Theorem \ref{stablemain} is closer to the i.i.d. case than to the case of a random walk among random obstacles as considered in \cite[Theorem 3]{BAMR05} where the limiting distributions are not stable laws, but infinite divisible distributions with Levy spectral functions that are not continuous. It seems as if the discrete character of the random walk is more decisive for that model than for ours which can be reduced to the i.i.d. case by virtue of an appropriate coarse-graining.\\
To get a feeling for the numbers involved we give them for the two examples mentioned above in the table below.\\
\newline
\noindent
\begin{tabular}{|l|l|l|l|}
 \hline
{\bf Distribution}&{\boldmath $ \varphi(x)$}&{\boldmath $ \log L_{\alpha}(t)$}&{\boldmath $\log B_{\alpha}(t)$} \\
\hline
Weibull& $x^\gamma,\,\gamma>1$  &$\left(\frac{\alpha t}{\gamma}\right)^{\gamma/(\gamma-1)}$& $t\left(\frac{\alpha t}{\gamma}\right)^{1/(\gamma-1)}-2d\kappa \alpha t+o(t)$\\
\hline
Double-exponential&$\exp\{x/\rho\}$& $\rho \alpha t$&$t\rho\log\rho \alpha t-\chi(\rho)\alpha t +o(t)$ \\
\hline
\end{tabular}\\
\newline
\noindent
Notice that in the Weibull case we have
 \[\log B_{\alpha}(t)=\frac{1}{\alpha}\left(\log\left<u(t,0)^\alpha\right>+\log|Q_{L_{\alpha}(t)}|\right)+o(t),\]
 see \cite{GS11}. Because of \eqref{GM98} and our considerations in Section~\ref{Stable limits} this relationship seems to be true in the double exponential case, as well.

\section{Stable limit laws}\label{Stable limits}

\noindent Let us explain our strategy of the proof of Theorem~\ref{stablemain}. We decompose the large box $Q_{L_{\alpha}(t)}$ into boxes $Q_{l(t)}^{\ssup i},\, i= 1,\cdots,\lfloor|Q_{L_{\alpha}(t)}|/|Q_{l(t)}|\rfloor$ of much smaller size.  
In each subbox we approximate $u$ by $u^{\ssup i}$, the solution with Dirichlet boundary conditions in $Q_{l(t)}^{\ssup i}$. 
In this way, we reduce the problem to the case of i.i.d.\ random variables. A spectral representation shows that $\sum_{x\in Q_{l(t)}^{\ssup i}}u^{\ssup i}(t,x)$ can be approximated by ${\rm e}^{t\lambda^{\ssup i}_1}$, where $\lambda^{\ssup i}_1$ is the principal Dirichlet eigenvalue of  $\Delta+\xi$ in  $Q_{l(t)}^{\ssup i}$. Then a classical result on stable limits for sums of $t$-dependent i.i.d.\ random variables yields the result.

Note that we cannot apply the results of \cite{BABM05} since they require the function $\varphi$ to be normalized regularly varying, which is for instance not true in the important case of double-exponential tails. An alternative approach could be to adopt the techniques from \cite{B06}.

Let us turn to the details. We first work on the $u^{\ssup i}$ and show in the end how to approximate $u$ by $u^{\ssup i}$. We assume that $Q_{l(t)}^{\ssup i}$ are translated copies of $Q_{l(t)}$. We consider the solution $u^{\ssup i}$ to the PAM in $Q_{l(t)}^{\ssup i}$ with Dirichlet boundary conditions, i.e.\ $\xi(x)=-\infty$ for all $x\notin Q_{l(t)}^{\ssup i}$, where $l(t)=\max\{t^2\log^2 t, H(4t)\}$. The corresponding Laplacian will be denoted $\Delta_{Q_{l(t)}^{\ssup i}}^0$.
The Feynman-Kac representation of $u^{\ssup i}$ reads
\[u^{\ssup i}(t,x)=\E_x\exp\Bigg\{\int\limits_0^\infty \xi(X_s)\,\text{d}s\Bigg\}\mathbbm{1}_{\tau_{(Q_{l(t)}^{\ssup i})^c}>t},\qquad(t,x)\in[0,\infty)\times Q_{l(t)}^{\ssup i}.\]
By $\tau_U:=\text{inf}\left\{t>0\colon X_t\in U\right\}$ we denote the first hitting time of a set $U$ by a random walk $X$. Let $\lambda_1^{\ssup i},\cdots,\lambda_{|Q_{l(t)}|}^{\ssup i}$ be the order statistics of the eigenvalues of the Anderson Hamiltonian $\Delta_{Q_{l(t)}^{\ssup i}}^0+\xi$ and $e_1^{\ssup i},\cdots,e_{|Q_{l(t)}|}^{\ssup i}$ be the corresponding orthonormal basis. Then we have the following spectral representation
\begin{equation}\label{spectral}
 \sum\limits_{x\in Q_{l(t)}^{\ssup i}}u^{\ssup i}(t,x)=\sum\limits_{x,y\in Q_{l(t)}^{\ssup i}}\sum\limits_{k=1}^{| Q_{l(t)}|} {\rm e}^{\lambda_k^{\ssup i} t}e_k^{\ssup i}(x)e_k^{\ssup i}(y),\qquad t\in[0,\infty).
\end{equation}
For simplicity we have suppressed the time dependence of the eigenvalues and eigenvectors that arises because the boxes are time dependent.
From Parseval's inequality, the fact that $l(t)$ is of subexponential order and the proof of Theorem 2.2 in \cite{GM98} it follows that there exists $\widetilde{\varepsilon}^{\ssup i}(t)=\widetilde{\varepsilon}^{\ssup i}(\xi,t)=o(1)$  such that
\[\sum\limits_{x\in Q_{l(t)}^{\ssup i}}u^{\ssup i}(t,x)={\rm e}^{t\mu_t^{\ssup i}},\qquad\mbox{where}\qquad \mu_t^{\ssup i}=\mu_t^{\ssup i}(\xi)=\lambda_1^{\ssup i}+\widetilde{\varepsilon}^{\ssup i}(t).\]
Sometimes we will write $\mu_t$ instead of $\mu_t^1$, $\lambda_1$ for $\lambda_1^{\ssup i}$ and $\widetilde \varepsilon(t)$ for $\widetilde \varepsilon^{\ssup i}(t)$.

\begin{remark}
The above already implies that for $\log r(t)=o\big(H(t)\big)$ the quenched setting is prominent in the following sense,
\[\lim\limits_{t\to\infty}\frac{\log u(t,0)}{\log\sum\limits_{x\in Q_{r(t)}}u(t,x)}=1,\qquad \text{a.s.}
\]
\end{remark} 
\noindent
In the next lemma we show how the distributions of $\mu_t$ and $\xi(0)$ are linked.

\begin{lemma}\label{trick1}
Let Assumption F be satisfied. Then for all functions $h$ with $\lim\limits_{t\to\infty}|Q_{l(t)}|\mathbf{P}\big(\xi(0)>h(t)\big)=0$ there exists $\varepsilon(t)=\varepsilon(\xi,t)=o(1)$ such that,
\[ \mathbf{P}\Big(\mu_{t}> h(t)\Big)\sim|Q_{l(t)}|\mathbf{P}\big(\xi(0)>h(t)+\chi-\varepsilon(t)\big),\qquad t\to\infty. \]
\end{lemma}

\begin{proof}
In \cite[Proof of Theorem 2.16]{GM98} the authors show that the first eigenvalue of $\Delta_{Q_{l(t)}}^0+\xi$ satisfies
\[\lambda_1=\xi^{\ssup 1}_{Q_{l(t)}}-\chi+\bar{\varepsilon}(t),\]
 with $\bar{\varepsilon}(t)=\bar{\varepsilon}(\xi,t)=o(1)$. Let \[\varepsilon(t):=\widetilde{\varepsilon}(t)+\bar{\varepsilon}(t).\]
 Then
\begin{align*}
 \mathbf{P}\Big(\mu_{t}>h(t)\Big)&=\mathbf{P}\Big(\xi^{\ssup 1}_{Q_{l(t)}}>h(t)+\chi-\varepsilon(t)\Big)\\
&=1-\bigg(1-\mathbf{P}\Big(\xi(0)>h(t)+\chi-\varepsilon(t)\Big)\bigg)^{|Q_{l(t)}|}\\
&\sim 1-\exp\Big\{-|Q_{l(t)}|\mathbf{P}\Big(\xi(0)>h(t)+\chi-\varepsilon(t)\Big)\Big\}\\
&\sim |Q_{l(t)}|\mathbf{P}\Big(\xi(0)>h(t)+\chi-\varepsilon(t)\Big),\qquad t\to\infty.
\end{align*}
 In the third line we use L'Hopital's rule.
\end{proof}
\noindent
Let 
\[\widetilde{\varphi}_t(x)=-\log\mathbf{P}\left(\mu_t>x\right),\]
and $\widetilde{h}_t=h_t+\chi+o(1)$, where the error term is chosen in a suitable way. In particular it is chosen such that if $\widetilde{\varphi}_t$ is ultimately convex then $\widetilde{h}_t$ is the unique solution to
\[\sup\limits_{h\in(0,\infty)}\left(th-\widetilde{\varphi}_t(h)\right)=t\widetilde{h}_t-\widetilde{\varphi}_t\left(\widetilde{h}_t\right).\]
Then, an application of the Laplace method yields
\[\left<u(t,0)\right>\sim\left<{u}^{\ssup i}(t,0)\right>\sim t\int\limits_{0}^{\infty}\exp\{th-\widetilde{\varphi}(h)\}\,\mathrm{d}h=\exp\{\big[t\widetilde{h}_t-\widetilde{\varphi}\big(\widetilde{h}_t\big)\big]\big(1+o(1)\big)\}.\]
The first asymptotics follow from \cite[Proposition 7]{GS11}.
Hence, we obtain together with Lemma~\ref{trick1} that 
\[
\log B_{\alpha}(t)=t\widetilde{h}_{\alpha t}=\frac{1}{\alpha}\left(\log\left<u(t,0)^\alpha\right>+\log|Q_{L_{\alpha}(t)}|\right)\left(1+o(1)\right).\]
\noindent
To prove convergence of $\sum_{i\colon Q_{l(t)}^{\ssup i}\subset Q_{L_{\alpha}(t)}}\big(\mathrm{e}^{t\mu_t^{\ssup i}}-\widetilde A(t)\big)/B_{\alpha}(t)$, as $t\to\infty$, to an infinitely divisible distribution with characteristic function equal to
\begin{equation}\label{stablechar}
 \phi(u)=\exp\bigg\{iau-\frac{\sigma^2u^2}{2}+\int\limits_{|x|>0}\bigg(\mathrm{e}^{iux}-1-\frac{iux}{1+x^2}\bigg)\,\mathrm{d}\widetilde{L}(x)     \bigg\},
\end{equation}
 we have to verify the following condition (see \cite[Chapter IV]{P75}).\\
\underline{Condition P}:
\begin{enumerate}
 \item  Condition of infinite smallness:
\[\lim\limits_{t\to\infty}\max\limits_{i\colon Q_{l(t)}^{\ssup i}\subset Q_{L_{\alpha}(t)}}\mathbf{P}\bigg(\frac{\mathrm{e}^{t\mu_t^{\ssup i}}}{B_{\alpha}(t)}\geq\varepsilon\bigg)=0,\qquad \varepsilon>0.\]     \item In all points $x$ of continuity, the function $\widetilde{L}$ satisfies:
\[\widetilde{L}(x)=-\lim_{t\to\infty}\frac{|Q_{L_{\alpha}(t)}|}{|Q_{l(t)}|}\mathbf{P}\bigg(\frac{\mathrm{e}^{t\mu_t}}{B_{\alpha}(t)}>x\bigg).\]
 \item  The constant $\sigma^2$ satisfies:
\begin{align*}
 \sigma^2=&\lim_{\tau\to 0}\limsup_{t\to\infty}\frac{|Q_{L_{\alpha}(t)}|}{|Q_{l(t)}|}\text{Var}\bigg(\frac{\mathrm{e}^{t\mu_t}}{B_{\alpha}(t)}\mathbbm{1}_{\frac{\mathrm{e}^{t\mu_t}}{B_{\alpha}(t)}\leq\tau}\bigg)\\
=&\lim_{\tau\to 0}\liminf_{t\to\infty}\frac{|Q_{L_{\alpha}(t)}|}{|Q_{l(t)}|}\text{Var}\bigg(\frac{\mathrm{e}^{t\mu_t}}{B_{\alpha}(t)}\mathbbm{1}_{\frac{\mathrm{e}^{t\mu_t}}{B_{\alpha}(t)}\leq\tau}\bigg).
\end{align*}
\item For every $\tau>0$ the constant $a$ satisfies:
\[\lim_{t\to\infty}\bigg\{\frac{|Q_{L_{\alpha}(t)}|}{|Q_{l(t)}|}\left<\frac{\mathrm{e}^{t\mu_t}}{B_{\alpha}(t)}\mathbbm{1}_{\frac{\mathrm{e}^{t\mu_t}}{B_{\alpha}(t)}\leq\tau}\right>-\frac{\widetilde{A}(t)}{B_{\alpha}(t)}\bigg\}=a+\int\limits_0^\tau \frac{x^3}{1+x^2}\,{\rm d}L(x) -\int\limits_\tau^\infty \frac{x}{1+x^2}\,{\rm d}L(x).\]
\end{enumerate}

\noindent
Items i) and ii) will follow from the next lemma, and iii) and iv) from the next proposition.

\begin{lemma}\label{alpha}
Let Assumption F be satisfied and $\varphi\in C^2$ be ultimately convex. Then
 \[\lim_{t\to\infty}\frac{|Q_{L_{\alpha}(t)}|}{|Q_{l(t)}|}\mathbf{P}\bigg(\mu_t>\frac{\log B_{\alpha}(t)}{t}+\frac{\log x}{t}\bigg)=x^{-\alpha}.\]
\end{lemma}

\begin{proof}
Lemma~\ref{trick1} and a first order Taylor expansion yield
\begin{align*}
 &\mathbf{P}\bigg(\mu_t>\frac{\log B_{\alpha}(t)}{t}+\frac{\log x}{t}\bigg)\sim|Q_{l(t)}|\mathbf{P}\bigg(\xi(0)>\frac{\log B_{\alpha}(t)}{t}+\frac{\log x}{t}+\chi+\varepsilon(t)\bigg)\\
=&\exp\bigg\{\log |Q_{l(t)}|-\varphi\bigg(\frac{\log B_{\alpha}(t)}{t}+\chi+o(1)\bigg)-\varphi'\bigg(\frac{\log B_{\alpha}(t)}{t}+\chi+o(1)\bigg)\frac{\log x}{t}+o(1)\bigg\}.
\end{align*}
Since $\varphi$ is ultimately convex and $\xi$ is unbounded from above we find that $\varphi''(h_{\alpha t})=1/h_{\alpha t}'=o(t^2).$ From this we can conclude that the error term in the Taylor expansion above vanishes asymptotically.
Moreover, by our choice of $l(t)$ and $B_{\alpha}(t)$ it follows that 
 \[\log |Q_{L_{\alpha}(t)}|=\varphi\bigg(\frac{\log B_{\alpha}(t)}{t}+\chi+o(1)\bigg)\qquad \text{and}\qquad
\lim_{t\to\infty}\frac{\varphi'\bigg(\frac{\log B_{\alpha}(t)}{t}+\chi+o(1)\bigg)}{t}=\alpha.\]

\end{proof}

\begin{proposition}\label{momente}
 Let Assumption F be satisfied and $\varphi\in C^2$ be ultimately convex. Then, for any $\tau>0$,
\begin{enumerate}
 \item if $p>\alpha$ then \[\lim\limits_{t\to\infty}\frac{|Q_{L_{\alpha}(t)}|}{|Q_{l(t)}|}\left<\frac{\mathrm{e}^{pt\mu_t}}{B_{\alpha}(t)^p}\mathbbm{1}_{\frac{\mathrm{e}^{t\mu_t}}{B_{\alpha}(t)}\leq\tau}\right>=\frac{\alpha}{p-\alpha}\tau^{p-\alpha}.\]
\item if $p<\alpha$ then \[\lim\limits_{t\to\infty}\frac{|Q_{L_{\alpha}(t)}|}{|Q_{l(t)}|}\left<\frac{\mathrm{e}^{pt\mu_t}}{B_{\alpha}(t)^p}\mathbbm{1}_{\frac{\mathrm{e}^{t\mu_t}}{B_{\alpha}(t)}>\tau}\right>=\frac{\alpha}{\alpha-p}\tau^{p-\alpha}.\]
\item if $p=\alpha$ then
\[\lim\limits_{t\to\infty}\frac{|Q_{L_{\alpha}(t)}|}{|Q_{l(t)}|}\left<\frac{\mathrm{e}^{pt\mu_t}}{B_{\alpha}(t)^p}\left(\mathbbm{1}_{\frac{\mathrm{e}^{t\mu_t}}{B_{\alpha}(t)}\leq\tau}-\mathbbm{1}_{\frac{\mathrm{e}^{t\mu_t}}{B_{\alpha}(t)}\leq1}\right)\right>=\alpha\log\tau.\]
\end{enumerate}

\end{proposition}

\begin{proof}
 \begin{enumerate}
  \item Integration by parts yields
\begin{align*}
\left< \mathrm{e}^{pt\mu_t}\mathbbm{1}_{\frac{\mathrm{e}^{t\mu_t}}{B_{\alpha}(t)}\leq\tau}\right>
=& \int\limits_0^{\widetilde{h}_{\alpha t}+\frac{\log\tau}{t}} \mathrm{e}^{tpx}\, \text{d}\left(1-\bar{F}_{\mu_t}(x)\right)\\
=&-\left[\mathrm{e}^{tpx-\widetilde{\varphi}(x)}\right]_{x=0}^{x=\widetilde{h}_{\alpha t}+\frac{\log\tau}{t}}+pt\int\limits_0^{\widetilde{h}_{\alpha t}+\frac{\log\tau}{t}} \mathrm{e}^{tpx-\widetilde{\varphi}(x)}\, \text{d}x.
\end{align*}
Here $\bar{F}_{\mu_t}$ denotes the tail distribution function of $\mu_t$.
Similarly as in the proof of Lemma~\ref{alpha} we find with the help of a first order Taylor expansion of $\varphi$ that uniformly in $\tau$,
\[\widetilde{\varphi}\big(\widetilde{h}_{\alpha t}+\frac{\log\tau}{t}\big)\sim\widetilde{\varphi}(\widetilde{h}_{\alpha t})-\alpha\log\tau,\qquad t\to\infty. 
\]
Substituting $x=\widetilde{h}_{\alpha t}+\frac{\log\tau}{t}u$ for $\tau\neq1$ we find that 
$$
\begin{aligned}
pt\int\limits_0^{\widetilde{h}_{\alpha t}+\frac{\log\tau}{t}} \mathrm{e}^{tpx-\widetilde{\varphi}(x)}\,\text{d}x
&\sim \mathrm{e}^{pt\widetilde{h}_{\alpha t}-\widetilde{\varphi}(\widetilde{h}_{\alpha t})}p\log\tau\int\limits^1_{-\infty\cdot\text{sign} \log\tau} \mathrm{e}^{u(p-\alpha)\log\tau }\,\text{d}u\\ &\sim\frac{p}{p-\alpha}\mathrm{e}^{pt\widetilde{h}_{\alpha t}-\widetilde{\varphi}(\widetilde{h}_{\alpha t})+(p-\alpha)\log\tau},\qquad t\to\infty.
\end{aligned}
$$
Altogether this proves the claim.
\item and iii) follow similarly.
 \end{enumerate}
\end{proof}
\noindent
Overall we find:
\begin{theorem}\label{mu}
Let Assumption F be satisfied and $\varphi\in C^2$ be ultimately convex. Then for $\alpha\in(0,2)$,
 \[\sum\limits_{i\colon Q_{l(t)}^{\ssup i}\subset Q_{L_{\alpha}(t)}}\frac{\mathrm{e}^{t\mu_t^{\ssup i}}-\widetilde{A}(t)}{B_{\alpha}(t)}\stackrel{t\to\infty}{\Longrightarrow}\mathcal{F_\alpha},\]
with
\[\widetilde{A}(t)=\begin{cases}
        0, &\text{if } \alpha\in(0,1),\\
\left<\mathrm{e}^{t\mu_t}\right>,  &\text{if } \alpha\in(1,2)\\
\left<\mathrm{e}^{t\mu_t}\mathbbm{1}_{\mu_t\leq 1}\right>, \quad &\text{if } \alpha=1.     \end{cases}
\]
\end{theorem}

\begin{proof}
Since the $u^{\ssup i}$ are i.i.d., the $\mu_t^{\ssup i}$ are as well. Hence, we have to check the four points of Condition P. Items i) and ii) follow from Lemma~\ref{alpha}. We find that $\widetilde{L}(x)=x^{-\alpha}$. It follows from Proposition~\ref{momente} that $\sigma^2=0$. Furthermore, Proposition~\ref{momente} together with \cite[Proposition 6.4]{BABM05} yields the constant $a$ from which we can deduce $\phi$. The stability of the limit law follows from \cite[Theorem IV.12]{P75} since $\sigma^2=0$ and $\widetilde{L}(x)=x^{-\alpha}$.
\end{proof}

\begin{remark}
 An infinitely divisible law with characteristic function as in \eqref{stablechar} is stable if and only if either $\widetilde{L}\equiv0$ or $\sigma^2=0$ and $\widetilde{L}(x)=cx^{-\alpha}$, $c>0$, $\alpha\in(0,2)$, see \cite[Theorem IV.12]{P75}.  
\end{remark}
\noindent
We extend the functions $u^{\ssup i}$ to a function $\widetilde u\colon Q_{L_\alpha(t)}\to[0,\infty)$ by putting $\widetilde u(t,x)= u^{\ssup i}(t,x)$ for $x\in Q_{l(t)}^{\ssup i}$. Now it remains to show that \\
\noindent
\begin{minipage}[t]{7.5cm}
\begin{align*}
 &\sum_{x\in Q_{L_{\alpha}(t)}}\frac{u(t,x)-A(t)}{B_{\alpha}(t)}\qquad \text {  and}\\
&\sum_{x\in Q_{L_{\alpha}(t)}}\frac{\widetilde{u}(t,x)-\widetilde{A}(t)/|Q_{l(t)}|}{B_{\alpha}(t)}\\
=&\sum_{i\colon Q_{l(t)}^{\ssup i}\subset Q_{L_{\alpha}(t)}}\frac{\exp\{t\mu_t^{\ssup i}\}-\widetilde{A}(t)}{B_{\alpha}(t)}
\end{align*}
 have the same $\alpha$-stable limit distribution. To this end let  $\mathcal{I}_t^c=\bigcup\limits_{i\colon Q_{l(t)}^{\ssup i}\subset Q_{L_{\alpha}(t)}} Q_{l(t)}^{\ssup i}\setminus Q_{l(t)(1-1/t)}^{\ssup i}$ 
and $\mathcal{I}_t=Q_{L_{\alpha}(t)}\setminus\mathcal{I}_t^c$. 
\end{minipage}
\begin{minipage}[t]{8cm}
\setlength{\unitlength}{0,4cm}
\begin{picture}(20,16)

\put(1,1){\line(1,0){16}}
\put(1,1){\line(0,1){16}}
\put(1,17){\line(1,0){16}} 
\put(17,1){\line(0,1){16}}
\put(1,9){\line(1,0){16}}
\put(9,1){\line(0,1){16}}

\dashline{0,5}(2,2)(2,8)
\dashline{0,5}(2,2)(8,2)
\dashline{0,5}(2,8)(8,8)
\dashline{0,5}(8,2)(8,8)

\dashline{0,5}(2,10)(2,16)
\dashline{0,5}(2,10)(8,10)
\dashline{0,5}(2,16)(8,16)
\dashline{0,5}(8,10)(8,16)

\dashline{0,5}(10,2)(16,2)
\dashline{0,5}(10,2)(10,8)
\dashline{0,5}(16,2)(16,8)
\dashline{0,5}(10,8)(16,8)

\dashline{0,5}(10,10)(10,16)
\dashline{0,5}(10,10)(16,10)
\dashline{0,5}(10,16)(16,16)
\dashline{0,5}(16,10)(16,16)

\put(1,1){\line(1,1){1}}
\put(2,1){\line(1,1){1}}
\put(3,1){\line(1,1){1}}
\put(4,1){\line(1,1){1}}
\put(5,1){\line(1,1){1}}
\put(6,1){\line(1,1){1}}
\put(7,1){\line(1,1){1}}
\put(8,1){\line(1,1){1}}
\put(9,1){\line(1,1){1}}
\put(10,1){\line(1,1){1}}
\put(11,1){\line(1,1){1}}
\put(12,1){\line(1,1){1}}
\put(13,1){\line(1,1){1}}
\put(14,1){\line(1,1){1}}
\put(15,1){\line(1,1){1}}
\put(16,1){\line(1,1){1}}

\put(1,8){\line(1,1){2}}
\put(2,8){\line(1,1){2}}
\put(3,8){\line(1,1){2}}
\put(4,8){\line(1,1){2}}
\put(5,8){\line(1,1){2}}
\put(6,8){\line(1,1){2}}
\put(7,8){\line(1,1){2}}
\put(8,8){\line(1,1){2}}
\put(9,8){\line(1,1){2}}
\put(10,8){\line(1,1){2}}
\put(11,8){\line(1,1){2}}
\put(12,8){\line(1,1){2}}
\put(13,8){\line(1,1){2}}
\put(14,8){\line(1,1){2}}
\put(15,8){\line(1,1){2}}
\put(1,9){\line(1,1){1}}

\put(1,16){\line(1,1){1}}
\put(2,16){\line(1,1){1}}
\put(3,16){\line(1,1){1}}
\put(4,16){\line(1,1){1}}
\put(5,16){\line(1,1){1}}
\put(6,16){\line(1,1){1}}
\put(7,16){\line(1,1){1}}
\put(8,16){\line(1,1){1}}
\put(9,16){\line(1,1){1}}
\put(10,16){\line(1,1){1}}
\put(11,16){\line(1,1){1}}
\put(12,16){\line(1,1){1}}
\put(13,16){\line(1,1){1}}
\put(14,16){\line(1,1){1}}
\put(15,16){\line(1,1){1}}
\put(16,16){\line(1,1){1}}

\put(1,1){\line(1,1){1}}
\put(1,2){\line(1,1){1}}
\put(1,3){\line(1,1){1}}
\put(1,4){\line(1,1){1}}
\put(1,5){\line(1,1){1}}
\put(1,6){\line(1,1){1}}
\put(1,7){\line(1,1){1}}
\put(1,8){\line(1,1){1}}
\put(1,9){\line(1,1){1}}
\put(1,10){\line(1,1){1}}
\put(1,11){\line(1,1){1}}
\put(1,12){\line(1,1){1}}
\put(1,13){\line(1,1){1}}
\put(1,14){\line(1,1){1}}
\put(1,15){\line(1,1){1}}
\put(1,16){\line(1,1){1}}

\put(16,1){\line(1,1){1}}
\put(16,2){\line(1,1){1}}
\put(16,3){\line(1,1){1}}
\put(16,4){\line(1,1){1}}
\put(16,5){\line(1,1){1}}
\put(16,6){\line(1,1){1}}
\put(16,7){\line(1,1){1}}
\put(16,8){\line(1,1){1}}
\put(16,9){\line(1,1){1}}
\put(16,10){\line(1,1){1}}
\put(16,11){\line(1,1){1}}
\put(16,12){\line(1,1){1}}
\put(16,13){\line(1,1){1}}
\put(16,14){\line(1,1){1}}
\put(16,15){\line(1,1){1}}
\put(16,16){\line(1,1){1}}

\put(8,1){\line(1,1){2}}
\put(8,2){\line(1,1){2}}
\put(8,3){\line(1,1){2}}
\put(8,4){\line(1,1){2}}
\put(8,5){\line(1,1){2}}
\put(8,6){\line(1,1){2}}
\put(8,7){\line(1,1){2}}
\put(8,8){\line(1,1){2}}
\put(8,9){\line(1,1){2}}
\put(8,10){\line(1,1){2}}
\put(8,11){\line(1,1){2}}
\put(8,12){\line(1,1){2}}
\put(8,13){\line(1,1){2}}
\put(8,14){\line(1,1){2}}
\put(8,15){\line(1,1){2}}

\put(18,10){\vector(0,1){7}}
\put(18,8){\vector(0,-1){7}}
\put(17,1){\makebox(3,16)[l]{\LARGE{$L_\alpha(t)$}}}

\put(11,6){\vector(0,1){3}}
\put(11,4){\vector(0,-1){3}}
\put(10,2){\makebox(6,6)[l]{\Large{$l(t)$}}}

\put(3,6){\vector(0,1){2}}
\put(3,4){\vector(0,-1){2}}
\put(2,3){\makebox(4,4)[l]{$l(t)(1-1/t)$}}
\put(8,8){\makebox(2,2)[c]{\huge{$\mathcal{I}_t^c$}}}

\end{picture}
\end{minipage}

Notice that
\[u(t,x)-\widetilde{u}(t,x)=\E_x\exp\Bigg\{\int\limits_0^\infty \xi(X_s)\,\text{d}s\Bigg\}\mathbbm{1}_{\tau_{(Q_{l(t)}^{\ssup i})^c}\leq t},\qquad(t,x)\in[0,\infty)\times Q_{l(t)}^{\ssup i}.\]
In the next lemma we show that those paths of the random walk in the Feynman-Kac formula that start in $\mathcal{I}_t$ and leave $Q_{l(t)}$ up to time $t$ are asymptotically negligible.
\begin{lemma}\label{uniform}
Almost surely,
\[\lim\limits_{t\to\infty}\sup\limits_{x\in Q_{l(t)(1-1/t)}}\E_x\exp\Bigg\{\int\limits_0^\infty \xi(X_s)\,\text{d}s\Bigg\}\mathbbm{1}_{\tau_{\left(Q_{l(t)}\right)^c}\leq t}=0.\]
\end{lemma}
\begin{proof}
We find that
\begin{align*}
&\sup\limits_{x\in Q_{l(t)(1-1/t)}}\E_x\exp\Bigg\{\int\limits_0^\infty \xi(X_s)\,\text{d}s\Bigg\}\mathbbm{1}_{\tau_{\left(Q_{l(t)}\right)^c}\leq t}\\
&\leq \exp\left\{t\sup\limits_{x\in Q_{l(t)}}\xi(x)\right\}\P_{0}\left(\tau_{\left(Q_{l(t)/t}\right)^c}\leq t\right)\\
&\leq2^{d+1}\exp\left\{t\cdot o\Big(\log\left(|Q_{l(t)}|\right)\Big)-|Q_{l(t)/t}|\log\left(\frac{|Q_{l(t)/t}|}{d\kappa t}\right)\right\}\stackrel{t\to\infty}{\longrightarrow} 0.
\end{align*}
In the last inequality we use \cite[Lemma 2.5 and Corollary 2.7]{GM98}.
\end{proof}
\begin{lemma}\label{cut}
 For all $\varepsilon>0$,
\begin{enumerate}
 \item if $\alpha\in(0,1]$ then
\[\lim_{t\to\infty}\mathbf{P}\Bigg(\frac{1}{B_{\alpha}(t)}\sum_{x\in Q_{L_{\alpha}(t)}}\big[u(t,x)-\widetilde{u}(t,x)\big]>\varepsilon\Bigg)=0.\]
\item if $\alpha\in[1,2)$ then
\[\lim_{t\to\infty}\mathbf{P}\Bigg(\frac{1}{B_{\alpha}(t)}\sum_{x\in Q_{L_{\alpha}(t)}}\big[u(t,x)-\left<u(t,x)\right>-\widetilde{u}(t,x)+\left<\widetilde{u}(t,x)\right>\big]>\varepsilon\Bigg)=0.\]
\item if $\alpha=1$ then
\[\lim_{t\to\infty}\mathbf{P}\Bigg(\sum_{x\in Q_{L_{\alpha}(t)}}\frac{u(t,x)-\left<u(t,x)\mathbbm{1}_{u(t,x)\leq B_{\alpha}(t)}\right>-\widetilde{u}(t,x)+\left<\widetilde{u}(t,x)\mathbbm{1}_{\widetilde{u}(t,x)\leq B_{\alpha}(t)}\right>}{B_{\alpha}(t)}>\varepsilon\Bigg)=0.\]
\end{enumerate}
\end{lemma}

\begin{proof}
\begin{enumerate}
\item 
From Lemma \ref{uniform} and the fact that $|\mathcal{I}_t|<B_\alpha(t)$ for all $t$ it follows that for $t\to\infty$,
\begin{equation*}
\mathbf{P}\Bigg(\frac{1}{B_{\alpha}(t)}\sum_{x\in Q_{L_{\alpha}(t)}} u(t,x)-\widetilde{u}(t,x)>\varepsilon\Bigg)
\sim \mathbf{P}\Bigg(\frac{1}{B_{\alpha}(t)}\sum_{x\in \mathcal{I}_t^c} u(t,x)-\widetilde{u}(t,x)>\varepsilon\Bigg).
\end{equation*}
By the definition of $B_{\alpha}(t)$ and by Markov's inequality it follows that
\begin{align*}
 \mathbf{P}\Bigg(\frac{1}{B_{\alpha}(t)}\sum_{x\in \mathcal{I}_t^c} u(t,x)-\widetilde{u}(t,x)>\varepsilon\Bigg)\leq&\mathbf{P}\Bigg(\sum_{x\in\mathcal{I}_t^c }\frac{u(t,x)}{|Q_{L_{\alpha}(t)}|^{1/\alpha}\left<u(t,0)^\alpha\right>^{1/\alpha}}>\varepsilon\Bigg)\\
\leq&\frac{1}{\varepsilon^\alpha}\frac{\Bigg<\bigg(\sum\limits_{x\in \mathcal{I}_t^c}u(t,x)\bigg)^\alpha\Bigg>}{|Q_{L_{\alpha}(t)}|\left<u(t,0)^\alpha\right>}\\
\leq&\frac{1}{\varepsilon^\alpha}\frac{|Q_{L_{\alpha}(t)}|}{|\mathcal{I}_t|}\frac{\left< u(t,0)^\alpha\right>}{|Q_{L_{\alpha}(t)}|\left<u(t,0)^\alpha\right>}\stackrel{t\to\infty}{\longrightarrow} 0.
\end{align*}
\item Similarly as in case $i)$ we find that asymptotically
\begin{align*}
 &\mathbf{P}\Bigg(\frac{1}{B_{\alpha}(t)}\sum_{x\in Q_{L_{\alpha}(t)}}\big[u(t,x)-\left<u(t,x)\right>-\widetilde{u}(t,x)+\left<\widetilde{u}(t,x)\right>\big]>\varepsilon\Bigg)\\
\leq&\frac{1}{\varepsilon^\alpha}\frac{\Bigg<\bigg(\sum\limits_{x\in \mathcal{I}_t^c}u(t,x)\bigg)^\alpha\Bigg>}{|Q_{L_{\alpha}(t)}|\left<u(t,0)^\alpha\right>}+o(1).
\end{align*}
Furthermore, we have
\begin{align*}
 &\Bigg<\bigg(\sum\limits_{x\in \mathcal{I}_t^c}u(t,x)\bigg)^\alpha\Bigg>\\
\leq&\left<\bigg(\sum\limits_{x\in \mathcal{I}_t^c}\Big(u(t,x)^2+\sum\limits_{\substack{y\in \mathcal{I}_t^c:\\|x-y|\leq l(t)/t}}u(t,x)u(t,y)+\sum\limits_{\substack{y\in \mathcal{I}_t^c:\\|x-y|> l(t)/t}}u(t,x)u(t,y)\Big)\bigg)^{\alpha/2}\right>\\
\leq& \sum\limits_{x\in \mathcal{I}_t^c}|l(t)(1-1/t)|\left<u(t,x)^\alpha\right>+\sum_{\substack{x,y\in \mathcal{I}_t^c:\\|x-y|>l(t)/t}}\left<\left(u(t,x)u(t,y)\right)^{\alpha/2}\right>.
\end{align*}
The first summand can be treated as in case i) whereas the second summand can be treated similarly as in the proof of Lemma~\ref{h4t}.
\item follows analogously.
\end{enumerate}
\end{proof}

\noindent
Now we are able to prove Theorem~\ref{stablemain}.

\begin{proof}[Proof of Theorem~\ref{stablemain}]
 We only consider the case $\alpha\in(0,1)$. The other cases follow similarly. It follows from Lemma~\ref{cut} that for every $\varepsilon>0$,
\[\lim_{t\to\infty}\mathbf{P}\Bigg(\sum\limits_{i\colon Q_{l(t)}^{\ssup i}\subset Q_{L_{\alpha}(t)}}\Bigg|\frac{\sum\limits_{x\in Q_{{l}(t)}^{\ssup i}}u(t,x)-\exp\{t\mu_t^{\ssup i}\}}{B_{\alpha}(t)}\Bigg|>\varepsilon\Bigg)=0,\]
while Theorem~\ref{mu} states that under the same conditions as in Theorem~\ref{stablemain},
\[\sum\limits_{i\colon Q_{l(t)}^{\ssup i}\subset Q_{L_{\alpha}(t)}}\frac{\mathrm{e}^{t\mu_t^{\ssup i}}}{B_{\alpha}(t)}\stackrel{t\to\infty}{\Longrightarrow}\mathcal{F_\alpha}.\]
Therefore, the claim follows from Slutzky's theorem.
\end{proof}

\begin{remark}
 We expect that a similar result as Theorem~\ref{stablemain} with the same stable limit distribution also holds for potential tails that are bounded from above as considered in \cite{BK01} and \cite{HKM06}. However, since in that case we do not have such a close link between between $\mu_t$ and $\xi^{\ssup 1}_{Q_{l(t)}}$ we cannot determine the distribution of $\mu_t$ and therefore $L_{\alpha}(t)=-\log\mathbf{P}(\mu_t>\widetilde{h}_t)$ cannot be made as explicit as under Assumption F.
\end{remark}

\section{Strong law of large numbers}\label{Strong law of large numbers}

Recall that $l(t)=\max\{t^2\log^2 t, H(4t)\}$ and that $x+Q_{l(t)}$ is the lattice box with centre $x$ and sidelength $l(t)$.

\begin{lemma}\label{h4t}
Let Assumption H be satisfied and $r(t)$ be chosen as in Theorem~\ref{slln}, then
\[\lim_{t\to\infty}\frac{1}{|Q_{r(t)}|^2}\sum_{\substack{x,y\in Q_{r(t)}:\\|x-y|>2l(t)}}\bigg(\frac{\left<u(t,x)u(t,y)\right>}{\left<u(t,0)\right>^2}-1\bigg)=0.\]
\end{lemma}
\begin{proof}
For $t>0$ and $x\in Q_{r(t)}$ let
\[u^{\ssup 1}(t,x)=\E_x\exp\Bigg\{\int\limits_0^\infty \xi(X_s)\,\text{d}s\Bigg\}\mathbbm{1}_{\tau_{(x+Q_{l(t)})^c}\geq t},\]
and
\[u(t,x,y)=\E_{x,y}\exp\Bigg\{\int\limits_0^\infty \xi(X_s)\,\text{d}s\Bigg\}\exp\Bigg\{\int\limits_0^\infty \xi(Y_s)\,\text{d}s\Bigg\}\mathbbm{1}_{\tau^X_{(x+Q_{l(t)})^c}<t\mbox{ or }\tau^Y_{(y+Q_{l(t)})^c}<t},\]
where $X$ and $Y$ are two independent random walks starting in $x$, $y$, respectively, $\E_{x,y}$ is their joint expectation, and $\tau^X_{A}$, $\tau^Y_{A}$ are their exit times from a set $A\subset \Z^d$, respectively.  
If $|x-y|>2l(t)$ then $u^{\ssup 1}(t,x)$ and $u^{\ssup 1}(t,y)$ are independent, and hence
\[
  \sum_{\substack{x,y\in Q_{r(t)}:\\|x-y|>2l(t)}}\bigg(\frac{\left<u(t,x)u(t,y)\right>}{\left<u(t,0)\right>^2}-1\bigg)
=\sum_{\substack{x,y\in Q_{r(t)}:\\|x-y|>2l(t)}}\frac{\left<u(t,x,y)\right>}{\left<u(t,0)\right>^2}.
\]
H\"older's inequality and \cite[Lemma 2.4 and Theorem 3.1]{GM90} yield for all $x,y\in Q_{r(t)}\setminus Q_{l(t)}$,
 \begin{align*}
\left<u(t,x,y)\right>\leq&\sqrt{\left<u(t,0)^4\right>2\P_x
\left(\tau_{(x+Q_{l(t)})^c}<t\right)}\\
\leq&\exp\left\{\frac{1}{2}(l(t)-l(t)\log l(t)+o(t)\right\}\stackrel{t\to\infty}{\longrightarrow} 0. 
\end{align*}
\end{proof}

\begin{proof}[Proof of Theorem~\ref{slln}]
By Chebyshev's inequality we find that for every $s>0$, 
\begin{align*}
 &\mathbf{P}\bigg(\sup_{t_n>s}\frac{1}{|Q_{r(t_n)}|}\sum_{x\in Q_{r(t_n)}}\Big(\frac{u(t_n,x)}{\left<u(t_n,0)\right>}-1\Big)>\varepsilon\bigg)\\
\leq&\sum_{t_n>s}\frac{1}{\varepsilon^2}\text{Var}\bigg(\frac{1}{|Q_{r(t_n)}|}\sum_{x\in Q_{r(t_n)}}\Big(\frac{u(t_n,x)}{\left<u(t_n,0)\right>}-1\Big)\bigg).
\end{align*}
As $t$ tends to infinity it follows with Lemma~\ref{h4t} that
\begin{align*}
 \text{Var}\bigg(\frac{1}{|Q_{r(t)}|}\sum_{x\in Q_{r(t)}}\Big(\frac{u(t,x)}{\left<u(t,0)\right>}-1\Big)\bigg)
&\sim\frac{1}{|Q_{r(t)}|^2}\sum_{\substack{x,y\in Q_{r(t)}:\\|x-y|<2l(t)}}\bigg(\frac{\left<u(t,x)u(t,y)\right>}{\left<u(t,0)\right>^2}-1\bigg)\\
&\sim\frac{1}{|Q_{r(t)}|}\sum_{x\in Q_{l(t)}}\bigg(\frac{\left<u(t,0)u(t,x)\right>}{\left<u(t,0)\right>^2}-1\bigg)\\
&\leq\frac{|Q_{l(t)}|}{|Q_{r(t)}|}\frac{\left<u(t,0)^2\right>}{\left<u(t,0)\right>^2}\\
&=\exp\{-\log|Q_{r(t)}|+H(2t)-2H(t)+o(t)\}.
\end{align*}
The last asymptotics are due to \eqref{GM98}. Now the claim follows because for our choice of $r(t)$,
\[\lim\limits_{s\rightarrow\infty}\sum_{t_n>s}\exp\{-\log|Q_{r(t_n)}|+H(2t_n)-2H(t_n)+o(t_n)\}=0.
\]
\end{proof}

\subsection*{Acknowledgement}
We would like to thank Wolfgang K\"onig for various helpful comments on the first draft of this paper.


\begin{thebibliography}{10}

\bibitem[BABM05]{BABM05} {\sc G.~Ben Arous, L.V.~Bogachev} and {\sc S.A.~Molchanov},
Limit theorems for sums of random exponentials, {\it Probab. Theory Related Fields}
{\bf 132}, 579--612 (2005).

\bibitem[BAMR05]{BAMR05} {\sc G.~Ben Arous, S.~Molchanov} and {\sc A.~Ramirez}, Tansition from the annealed to the quenched asymptotics for a random walk on random obstacles, {\it Ann. Probab.} 
{\bf 33}, 2149--2187 (2005). 

\bibitem[BAMR07]{BAMR07} {\sc G.~Ben Arous, S.~Molchanov} and {\sc A.~Ramirez}, Transition asymptotics
for reaction-diffusion in random media, In Probability and Mathematical
Physics: A Volume in Honor of Stanislav Molchanov, AMS/CRM, 42, 1--40 (2007).

\bibitem[BK01]{BK01} {\sc M.~Biskup, W.~K\"onig}, Long-time tails for the parabolic Anderson model with bounded potential, {\it Ann. Probab.} {\bf29}, 636--682 (2001).

\bibitem[B06]{B06} {\sc L.~Bogachev} Limit laws for norms of IID samples with Weibull tails, {\it J. Theoret. Probab.} {\bf 19}, 849--873 (2006).

\bibitem[BKL02]{BKL02} {\sc A.~Bovier, I.~Kurkova} and {\sc M.~L\"owe}, Fluctuations of the free energy in the REM and the $p$-spin SK models, {\it Ann. Math. Statist.} {\bf 36}, 605--651 (2002).

\bibitem[CM07]{CM07} {\sc M.~Cranston, S.A.~Molchanov}, Quenched to annealed transition in the parabolic Anderson problem, {\it Probab. Theory Related Fields}
  {\bf 138}, 177--193 (2007).

\bibitem[GdH99]{GdH99} {\sc J.~G\"artner, F.~den Hollander}, Correlation structure of intermittency in the parabolic Anderson model, {\it Probab. Theory Related Fields} {\bf 114}, 1--54 (1999). 

\bibitem[GK05]{GK05} {\sc J.~G\"artner, W.~K\"onig}, The parabolic Anderson model, in: J.-D. Deuschel and A. Greven (Eds.), Interacting Stochastic Systems, 153--179, Springer (2005).

\bibitem[GKM07]{GKM07} {\sc J.~G\"artner, W.~K\"onig} and {\sc S.A.~Molchanov}, Geometric characterization of intermittency in the parabolic Anderson model, {\it Ann. Probab.} {\bf35}, 439--499 (2007).

\bibitem[GM90]{GM90} {\sc J.~G\"artner, S.A.~Molchanov}, Parabolic problems for the Anderson model: I. Intermittency and related topics, {\it Commun. Math. Phys.}
  {\bf 132}, 613--655 (1990).

\bibitem[GM98]{GM98} {\sc J.~G\"artner, S.A.~Molchanov}, Parabolic problems for the Anderson model: II.  Second-order asymptotics and structure of high peaks, {\it Probab. Theory Related Fields}
  {\bf 111}, 17--55 (1998).

\bibitem[GS11]{GS11} {\sc J.~G\"artner, A.~Schnitzler}, Time correlations for the parabolic Anderson model, {\it Electron. J. Probab.} {\bf 16}, 1519--1548 (2011).

\bibitem[HKM06]{HKM06} {\sc R.~van der Hofstad, W.~K\"onig} and {\sc P.~M\"orters}, The universality classes in the parabolic Anderson model, {\it Commun. Math. Phys.} {\bf267}, 307--353 (2006).

\bibitem[J10]{J10} {\sc A.~Jan{\ss}en} Limit laws for power sums and norms of i.i.d. samples, {\it Probab. Theory Related Fields}
  {\bf 146}, 515--533 (2010).

\bibitem[M94]{M94} {\sc S.A.~Molchanov}, Lectures on Random Media, {\it Lecture Notes in Math.}
  {\bf 1581}, 242--411, Springer (1994).

\bibitem[P75]{P75} {\sc V.V.~Petrov}, Sums of Independent Random Variables, Springer (1975). 

\end{thebibliography}
\end{document}